\documentclass[11pt]{amsart}
\usepackage{amssymb}
\newtheorem{prop}{Proposition}

\usepackage{color}
\let\s\mathsf
\let\Cal\mathcal
\def\Frac#1#2{{\textstyle #1\over\textstyle #2}}%

\def\Frac#1#2{{\textstyle #1\over\textstyle #2}}%

\dedicatory{To the memory of Sasha Anan'in}
\title{Integral criteria of hyperbolicity
for graphs and groups}

\author{Victor Gerasimov}
\address{Victor Guerassimov (Gerasimov), Departamento de Matem\'atica,
Universidade Federal de Minas Gerais,
Av. Ant\^onio Carlos, 6627/
CEP: 31270-901. Caixa Postal: 702
Belo Horizonte, MG,
Brazil}
\email{victor@mat.ufmg.br}

\author{Leonid Potyagailo\\ \\ $\text{\today}$}
\address{ Leonid Potyagailo, UFR de Math\'ematiques, Universit\'e
de Lille, 59655 Villeneuve d'Ascq cedex, France}
\email{potyag@math.univ-lille1.fr}

\thanks{This work was
supported in part by the Labex CEMPI (ANR-11-LABX-0007-01) and by
the Brazilian-French Network in
Mathematics.}

\thanks{At the request of the publisher, the authors declare that
they have no relevant financial or non-financial interests to disclose.}


\begin{document}
\begin{abstract}
We establish three criteria of hyperbolicity of a graph in terms of
``average width of geodesic bigons''.
In particular we prove that if the ratio of the Van Kampen area of a geodesic bigon $\beta$ and the length of $\beta$
in the Cayley graph of a finitely presented group $G$
is bounded above then $G$ is hyperbolic.


\end{abstract}
\date{\today}
\maketitle
\markboth{V.~Gerasimov, L.~Potyagailo, {\sl Integral criteria fo hyperbolicity}\quad(\today)}{V.~Gerasimov, L.~Potyagailo,
{\sl Integral criteria fo hyperbolicity}\quad(\today)}

\def\mpPath{}
\section{Introduction}
\subsection{Characteristic properties of hyperbolicity}
In 1987 M. Gromov introduced the definition of hyperbolic metric space \cite{Gr87}.

The hyperbolic metric spaces proved to be useful in many problems of geometry and group theory
and became one of the main objects of the geometrc group theory.

Subsequently many different characterizations of the hyperbolicity property were established.
Each of these properties can be taken as a definition of the hyperbolicity.

Sometimes a property characterizes the hyperbolicity only under certain restrictions on
a metric space. For example, it may be required that the metric is a path-metric
or geodesic.

There are also properties characterising \textit{hyperbolic graphs} within
the class of all connected graphs.

A \textit{graph} can be defined as a specific path-metric space
whose metric is either continuous or discrete,
with the values in the set $\Bbb N{=}\{0,1,2,\dots\}$ of positive integers.
In particular, P. Papasoglu proved
that the hyperbolicity of a connected graph follows from the
uniform thinness of the geodesic bigons \cite[Theorem 1.4]{Pap}.

In the current paper mixing the ideas of \cite{Pap} with ideas related to the
random walks on groups, we establish a new criterion
of hyperbolicity of a graph.
\subsection{Problems of checking hyperbolicity}
Despite the abundance of criteria of hyperbolicity,
sometimes to prove that a concrete space is hyperbolic remains a difficult
problem.
These difficulties motivate the search for new criteria.

The thin bigon criterion developed by P. Papasoglu turned out
to be useful. In particular it was applied
to show that
a group is hyperbolic if it is ``strongly geodesically automatic'' \cite{Pap}.

At the same time there are many cases when the hyperbolicity of a metric space remains a largely open problem.
\vskip3pt
1. The well-known conjecture atributed to Gersten \cite{Gers} states that a one-relator group
is hyperbolic if and only if it does not contain Baumslag-Solitar subgroups.\par
2. Does the \it Ancona property\rm\footnote{This means that there exists $\varepsilon >0$ such that for every geodesic segment $[x, z]$ and any $y\in (x,z)$ the probability that a random path from $x$ to $z$ passes through $y$ is greater \hfill\break than $\varepsilon$, see {\cite{GGPY}}.} of  a random walk
defined on the Cayley graph of
a finitely generated group $G$, imply that $G$ is hyperbolic?
A. Ancona used this property in his proof that the Martin boundary of a finitely supported random walk
on an arbitrary hyperbolic group $G$ coincides with the Gromov's boundary $\partial G$ \cite{Anc}?

The intriguing open question 2
is related to the problem of characterization of the class of hyperbolic groups
in terms of random walks (see \cite{GGPY} for generalizations for wider classes of groups).
Note that certain classes of groups (e.g. amenable groups) are known to possess such characterizations.

This problem was our main motivation  for the current paper.
\subsection{Integral criteria for hyperbolicity}
Many known characteristic properties of the hyperbolicity claim that
some function of several variables defined on a metric space is bounded above.
Here are few examples of such functions: $(a)$
the ``width'' (in some exact sense) of a geodesic
triangle (or of a geodesic bigon); $(b)$ the ratio of the ``area'' of a closed rectifiable path and
the length of the path.

Sometimes instead of boundness it is required another property of asymptotic behaviour of a function.

The variety of hyperbolicity criteria can be divided in ``pure geometric'',
such as boundness of the width of triangles or bigons; or ``integral'' which are given
by averagings of geometric functions.
An example of an integral one
is the well-known criterion of linearity (or subquadraticity) of the Dehn function.

Since the hyperbolicity  is a large scale property of a metric space, the integral and probabilistic criteria are useful. In the current paper we provide several ``integral'' generalizations of
the hyperbolicity criterion given in \cite[Theorem 1.4]{Pap},
averaging the width function of a geodesic bigon.

One of the corollaries of our results is the following (see Theorem 2 below): if the ratio $\frac{\displaystyle\mathsf{area}(\beta)}{\displaystyle|\beta|}$ is bounded above then the group is hyperbolic.
Here $\beta$ ranges over the set of geodesic bigons in the Cayley graph of a finitely presented group, $|\beta|$
denotes the length of the bigon and $\s{area}(\beta)$ is the minimal number of cells in Van Kampen diagrams
for $\beta$ (see Subsection 1.7 below).
\subsection{Geodesic bigons}
In a metric space, a (geodesic) \it bigon \rm is a pair
of geodesic paths with coinciding initial and terminal points.
For example, since any $\mathsf{CAT}(0)$ space $S$ is \it unigeodesic\rm, i.e. a geodesic segment
joining two points is unique,
the geodesic bigons are trivial in $S$.
Of course the absence of nontrivial geodesic bigons does not imply hyperbolicity.

However, a rather unexpected result of \cite{Pap} states that if $S$ is a \textbf{metric graph}
then the hyperbolicity of $S$ follows from the existence of a uniform bound for the ``width'' of geodesic bigons.
It is worth noting that, in \cite{Pap} the word `graph' means a one-dimentional connected $\s{CW}$-complex
whose edges are isometric to the interval $[0,1]{\subset}\Bbb R$.
The initial and terminal points of a geodesic bigon do not necessarily
belong to the set of the vertices of the graph $S$.

In the current paper, speaking on bigons, we omit the adjective `geodesic' since
we do not consider non-geodesic bigons.
\subsection{Discrete version of bigons}
Unlike \cite{Pap} we prefer to work with  ``discrete'' graphs
that is the metric spaces whose distance function takes the values in $\Bbb N$.
Such a space $S$ is interpreted as a graph by declaring that
two \textbf{distinct} points $p,q$ of $S$ are joined by an edge if and only if the distance $|p$-$q|$ is equal
to 1.  This restriction implies the following modification of the notions of path and of bigon.
\vskip3pt
\bf Convention\rm. In this paper the sign `$\leftrightharpoons$' means `is equal by definition'.
By $\Bbb N_{\Cal C}$ we denote the set $\{i{\in}\Bbb N:i$ has property $\Cal C\}$.
For example, $\Bbb N_{<x}{\leftrightharpoons}\{i{\in}\mathbb N:i<x\}$.
\vskip3pt
\bf Definition 1\rm. A \textit{path} in a graph $S$ is a
distance non-increasing map of the form $\Bbb N_{\leqslant l}\overset\alpha\to S$
where the number $l{\in}\Bbb N$ is  the \textit{length} $|\alpha|$ of the path $\alpha$.
A path $\Bbb N_{\leqslant l}\overset\alpha\to S$ is \textit{geodesic} if $l{=}|\alpha(0)$-$\alpha(l)|$.
We postulate that $S$ is connected. It follows that every two points of $S$ can be joined by a geodesic path.

\vskip3pt

Denote $\s{Geo}S\leftrightharpoons\{$the geodesic paths in $S\}$.
For \textbf{distinct} paths $\alpha,\alpha'{\in}\s{Geo}S$,
the \textit{distance} $|\alpha$-$\alpha'|$ is the number
$\s{max}\{1,|\alpha(0)$-$\alpha'(0)|{+}|\alpha(|\alpha|)$-$\alpha'(|\alpha'|)|\}$.
We consider $\s{Geo}S$ as a discrete metric graph in the same way as the graph $S$ itself.

For a pair $\beta{=}\{\alpha,\alpha'\}{\subset}\s{Geo}S$
the \textit{distance function}
is the function\hfil\penalty-10000
 $\Bbb N_{\leqslant l}{\ni}i\mapsto\s w_\beta(i){=}|\alpha(i)$-$\alpha'(i)|$
where $l{=}\s{min}\{|\alpha|,|\alpha'|\}$.
\vskip3pt
\bf Definition 2\rm. A \textit{bigon} in $S$
is a pair $\beta{=}\{\alpha,\alpha'\}$ in $\s{Geo}S$ \textbf{of the same length} $l$ such that $|\alpha$-$\alpha'|{=}1$.
The number $l$ is called \textit{length} $|\beta|$ of the bigon.
We denote $\s{Bg}S\leftrightharpoons\{$the bigons in $S\}$.
\vskip3pt
After this ``discrete'' formalization the Theorem 1.4 of \cite{Pap} can be stated as follows:\vskip3pt
$(*)$\quad\sl If, for a connected graph $S$, the
set $\{\s w_\beta(i):\beta{\in}\s{Bg}S,i{\in}\Bbb N_{\leqslant|\beta|}\}$
 is bounded then $S$ is hyperbolic\rm.
\vskip3pt
In other words the unform thinness of bigons implies the hyperbolicity.
\subsection{Thinness of bigons almost surely}
We would like to soften the hypothesis of the uniform thinness of bigons in the sense of \cite{Pap}.

Let us say that, in a graph $S$, \it the bigons are thin almost surely\rm,
if, for every $\varepsilon{>}0$ there exists $x{\in}\Bbb N$ such that
$\#\{i{\in}\Bbb N:\s{w}_\beta(i){>}x\}\leqslant\varepsilon{\cdot}|\beta|$
(here `$\#$' means the cardinality).

Denote this condition by $\s{TAS}$ (``thin almost surely'').
It can be stated in the following equivalent form:

$\s{TAS}:$\quad\sl The family of functions
$x\mapsto\Frac{\#\s{w}_\beta^{-1}(\Bbb R_{>x})}{|\beta|}$  converges
to zero as $x\to\infty$ \bf uniformly \sl with respect to $\beta{\in}\s{Bg}S$\rm.
\vskip3pt

Here, of course, $\s{w}_\beta^{-1}(\Bbb R_{>x}){=}\{i{\in}\Bbb N:\s{w}_\beta(i){>}x\}$.

One of the main results of the current paper is that $\s{TAS}$ implies the hyperbolicity.
It will be an immediate corollary of Theorem 1 (see Subsection 1.8).
\subsection{Dehn-like function for bigons}
As an application of our main result
we prove the following criterion of hyperbolicity.\par
Consider a finite group presentation $\Cal P{=}(X;\Cal R)$, where $X$ is a finite genera\-ting set
and $\Cal R$ is a finite symmetric set of cyclically reduced defining relations.
Let $S$ denote the Cayley graph corresponding to $\Cal P$ and let $\gamma\mapsto\s a(\gamma)$
 denote the Van Kampen area
of a closed path $\gamma$ in $S$ (see Section 6 for details).
Thus $\s a(\gamma)$ is the minimal number of cells in a Van Kampen
diagrams with boundary path $\gamma$.
For a bigon $\beta{=}\{\alpha,\alpha'\}{\in}\s{Bg}S$ we denote by $\s a(\beta)$
the value $\s a(\gamma)$ where $\gamma$ is a closed path that contains the concatenation $\alpha'{\circ}\alpha^{-1}$
or $\alpha^{-1}{\circ}\alpha'$.
\vskip3pt
\sc Theorem 2\sl.
If a group $G$ possesses a finite presentation
$\Cal P{=}(X;\Cal R)$
 such that
the set
\begin{equation}\label{vk}
\left\{\Frac{\s a(\beta)}{|\beta|}:\beta{\in}\s{Bg}S\right\}\end{equation} is bounded above
then $G$ is hyperbolic\rm.
\subsection{The main result}
Our main technical result states that a condition rather weaker than $\s{TAS}$ implies hyperbolicity.

While $\s{TAS}$ means the uniform convergence of the family $x\mapsto\Frac{\#\s{w}_\beta^{-1}(\Bbb R_{>x})}{|\beta|}$ $(\beta{\in}\s{Bg}S)$
our weakened hypotheses postulates the existence of ``two first steps'' of the convergence.
\vskip5pt
\sc Theorem 1\sl. If, for a connected graph $S$, there exist $Y,Z{\in}\Bbb N$ such that:\hfil\penalty-10000
\vskip3pt
$\s A:$\quad$\s{sup}\left\{\Frac{\#\s{w}_\beta^{-1}(\Bbb N_{>Y})}{|\beta|}:\beta{\in}\s{Bg}S\right\}<1$;\hfil\penalty-10000
\vskip3pt
$\s B:$\quad$\s{sup}\left\{\Frac{\#\s{w}_\beta^{-1}(\Bbb N_{>Z})}{|\beta|}:\beta{\in}\s{Bg}S\right\}<\Frac1{4Y\overset{\vphantom.}{+}2}$,\hfil\penalty-10000
\vskip3pt
then $S$ is hyperbolic\rm.
\vskip7pt
Since $\s{TAS}$ implies $\s A$ and $\s B$ we have the following:
\vskip5pt
\sc Corollary\sl. The condition $\s{TAS}$ imples the hyperbolicity of the graph $S$\rm. \vskip3pt
The proof of Theorem 1 consists in verification that the conditions $\s A$ and $\s B$ imply the uniform
thinness of bigons. The conclusion will then follow from the condition $(*)$.
\subsection{Acknowledgements}
The authors are thankful to the research grant LABEX CEMPI (ANR-11-LABX-0007-01) for
providing a support to Victor Gerasimov to visit the University of
Lille in Summer 2022 when the works on the project has been started.
The authors are also grateful to the Brazilian-French Network in
Mathematics for providing a support to Leonid Potyagailo to visit
Federal University of Belo Horizonte in December 2022 where the paper has been
completed.
 \section{Dense Value Lemma}
\subsection{Triangle inequalities for the width function}
 Let $S$ be a graph (in the sense of 1.5). Keeping the notations of the Introduction, for a path $\alpha{\in}
 \s{Geo}S$, denote by ${\s I}_\alpha$ the segment $[0,|\alpha|]{\subset}\Bbb R$.

\begin{prop}\label{prop1} Let $A,B{\in}\Bbb R$ and  $\alpha,\beta,\gamma{\in}\s{Geo}S$, $a{\in}\s I_\alpha{\cap}\s I_\beta{\cap}\s I_\gamma$. If\hfil\penalty-10000
$\Cal A{\leftrightharpoons}\Bbb N_{\leqslant a}{\cap}\s{w}_{\alpha,\beta}^{-1}(\Bbb R_{>A})$,
$\Cal B{\leftrightharpoons}\Bbb N_{\leqslant a}{\cap}\s{w}_{\beta,\gamma}^{-1}(\Bbb R_{>B})$,
$\Cal C{\leftrightharpoons}\Bbb N_{\leqslant a}{\cap}\s{w}_{\alpha,\gamma}^{-1}(\Bbb R_{>A+B})$
then\hfil\penalty-10000
 $\Cal{C{\subset}A{\cup}B}$ and $\#\Cal{C\leqslant\#A+\#B}$.\end{prop}
\begin{proof} Follows from definitions and the $\triangle$-inequality.\end{proof}

Using the induction on $n$ and Proposition \ref{prop1} we have:

\vskip5pt

{\bf Corollary.} \textsl{Let $\alpha_0,\dots,\alpha_n{\in}\s{Geo}S$ be paths whose lengths belong to the segment $[a,b]$ ($a{<}b$) and
$\#(\Bbb N_{\leqslant a}{\cap}\s{w}_{\alpha_{i-1},\alpha_i}^{-1}(\Bbb R_{>A_i}))\leqslant\varepsilon b$ for all $i{\in}\overline{1,n}$.
Then}
 $\#(\Bbb N_{\leqslant a}{\cap}\s{w}_{\alpha_0,\alpha_n}^{-1}(\Bbb R_{>A}))\leqslant n\varepsilon b$,
where $A{\leftrightharpoons}\sum_{i=1}^nA_i$.

 \subsection{Bands}
 A \textit{band} is an \textbf{ordered} pair $\beta=(\beta_0,\beta_1)$ of geodesic
paths in $S$ of the same length.
 The paths $\beta_i$ are called the \textit{sides} of $\beta$.
 The length $\vert\beta\vert$ of $\beta$ is the length of its sides.
 Denote $\s I_\beta{\leftrightharpoons}\s I_{\beta_0}{=}\s I_{\beta_1}$.
We consider the width function $\s w_\beta$ of a band $\beta$ defined
on the set $\Bbb N_{\leqslant|\beta|}$.

\begin{prop}\label{prop2}
Let $a_-,a_+,Z\in\Bbb N$ and let
 $\varepsilon{\in}\left(0,{1\over a}\right)$ where $a{=}a_-{+}a_+$. Assume that
every bigon $\{\alpha,\alpha'\}{\in}\s{Bg}S$ satisfies the condition:
\begin{equation}\label{EZ}
\tag*{$\s E(\alpha,\alpha')$}\#\s w_{\alpha,\alpha'}^{-1}(\Bbb N_{>Z})<\varepsilon{\cdot}|\alpha|.
\end{equation}
Then, for every $\nu{\in}(\varepsilon a,1)$, there exists $R{\in}\Bbb N$
such that if the length $l$ of a band $\beta$ is ${\geqslant}R$
and  $\s{w}_\beta(0){\leqslant}a_-,\s{w}_\beta(l){\leqslant}a_+$,
then
$$\#\s w_\beta^{-1}(\Bbb R_{>D})<\nu{\cdot}l$$
where $D{\leftrightharpoons}aZ{+}a_-$.\end{prop}
\begin{proof}
Let $\alpha{=}\alpha_0,\alpha_1,\dots,\alpha_a{=}\alpha'$ be a path in the graph $\s{Geo}S$,
 connecting the sides of the band $\beta{=}(\alpha,\alpha')$.
 \noindent Since the lengths of the neighboring paths  $\alpha_i$ differ by at most $1$, then
$|\alpha_i|{\in}\left[l{-}{a\over2},l{+}{a\over2}\right]$ for all $i{\in}\Bbb N_{\leqslant n}$.\par
\begin{picture}(0,0)(0,-10)
\put(300,-10){\makebox(0,0)[cb]{$\alpha{=}\alpha_0$}}
\put(275,-76){\makebox(0,0)[cc]{$\alpha'{=}\alpha_a$}}
\put(-14,-70){
\pdfximage{\mpPath 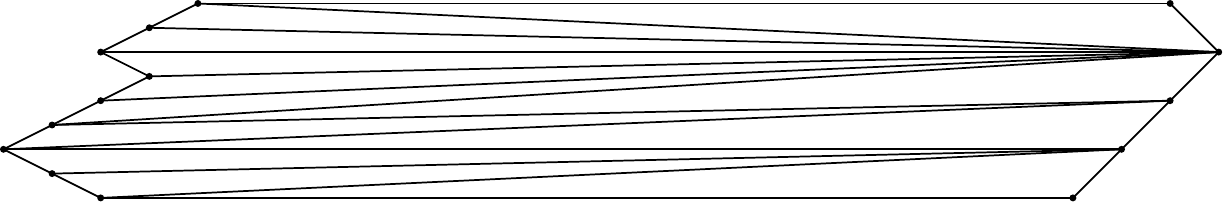}\pdfrefximage\pdflastximage}
\end{picture}\vskip72pt
So for every $i{\in}\Bbb N, 1{\leqslant}i{\leqslant}n$,
either $|\alpha_{i-1}|{=}|\alpha_i|$, and, by $\s E(\alpha_{i-1},\alpha_i)$
\begin{equation}\#(\Bbb N_{\leqslant l-(a/2)}{\cap}\s{w}_{\alpha_{i-1},\alpha_i}^{-1}(\Bbb R_{>Z}))\leqslant
\varepsilon{\cdot}|\alpha_i|\leqslant\varepsilon\left(l{+}{a\over2}\right),\end{equation}
or $||\alpha_{i}|{-}|\alpha_{i-1}||{=}1$ and there exists a bigon with one side equal to
$\alpha_{i}$ or $\alpha_{i-1}$ and the other side equal to a subpath of $\alpha_{i-1}$ or $\alpha_i$ respectively.
In the latter case we have
 \begin{equation}\label{c2}\#(\Bbb N_{\leqslant l-(a/2)}{\cap}\s{w}_{\alpha_{i-1},\alpha_i}^{-1}(\Bbb R_{>Z+1}))\leqslant
\varepsilon{\cdot}\s{max}\{|\alpha_{i-1}|,|\alpha_i|\}\leqslant\varepsilon\left(l{+}{a\over2}\right).\end{equation}
Indeed, $\s w_\delta\geqslant\s w_{\alpha_{i-1,\alpha_i}}{-}1$ for the bigon $\delta$ obtained by adding an edge
to $\alpha_{i-1}$ or $\alpha_i$. Then (\ref{c2}) follows from $\s E(\delta)$.

Applying the Corollary of Proposition \ref{prop1} and  using $(1)$ and $(2)$, we obtain \hfil\penalty-10000
$\#\left(\Bbb N_{\leqslant l-(a/2)}{\cap}\s{w}_\beta^{-1}(\Bbb R_{>aZ+a_-})\right)\leqslant
\varepsilon{\cdot}a\left(l{+}{a\over2}\right)$. It implies that\hfil\penalty-10000
$\#\left(\s{w}_\beta^{-1}(\Bbb R_{>aZ+a_-})\right)\leqslant
\varepsilon{\cdot}a\left(l{+}{a\over 2}\right){+}{a\over2}<(\varepsilon a)l{+} a$ as
 $\varepsilon a<1$ and\hfil\penalty-10000
 $\#\Bbb N_{<a/2}\leqslant a/2$.

Since $\displaystyle\underset{l\rightarrow\infty}{\s{lim}}\frac{(\varepsilon a)l{+}a}l=\varepsilon a<1$,
there exists $R{\in}\Bbb N$
such that
$\displaystyle\frac{(\varepsilon a)l{+}a}l<\nu$ for all $l{\geqslant}R$.

Thus for $l{\geqslant}R$, we have $\#\left(\s{w}_\beta^{-1}(\Bbb R_{>aZ+a_-})\right)\leqslant\nu l$.\par
The Proposition is proved.\end{proof}

{\bf Lemma} ({\sf Dense Value}).
{\sl Let $a_-,a_+,a,\varepsilon,\nu,Z,R$ be as in Proposition \ref{prop2}.\hfil\penalty-10000
 Assume that $\s{E}(\alpha,\alpha')$ holds for every bigon $(\alpha,\alpha')$ in $S$. Set $D{=}aZ{+}a_-$.
Then, for every band $\beta$ of length $l\geqslant R$ with $\s w_\beta(0){\leqslant}a_-$, $\s w_\beta(l){\leqslant}a_+$,
there exists $p{\in}\Bbb N_{\leqslant D}$ such that}
$\displaystyle\frac{\#\s{w}_\beta^{-1}(p)}l\geqslant\varrho\leftrightharpoons\frac{1{-}\nu}{D{+}1}$.
\begin{proof} Consider a band $\beta$ satisfying the hypothesis.
 By Proposition \ref{prop2},\hfil\penalty-10000
 $\#\s w_\beta^{-1}(\Bbb N_{\leqslant D})\geqslant(1{-}\nu)l$.
Since $\#\Bbb N_{\leqslant D}=D{+}1$, at least one of the sets $\s w_\beta^{-1}p$ ($p{\in}\Bbb N_{\leqslant D}$)
has cardinality ${\geqslant}\Frac{(1{-}\nu)l}{D{+}1}=\varrho{\cdot}l$.
\end{proof}
\vskip3pt
Each such a value $p$ we call \it dense value \rm for the band $\beta$.

\section{\sc Rank and jumpers}
\vskip3pt
\subsection{Equilateral segments and their rank}
We continue considering a graph $S$ keeping the previous notations.

Let us call by \it fork \rm any band $\beta{=}(\gamma,\gamma')$ such that
$|\gamma(0)$-$\gamma'(0)|\leqslant1$.

In this section we fix a fork $\beta{=}\{\gamma,\gamma'\}$ and a value $p{\in}\Bbb N$.
(Further, $p$ will be a dense value for some band, but it is not essential in this section.)

The numbers from the set $\Cal P{\leftrightharpoons}\s{w}_\beta^{-1}p$ are called $p$\it-jumpers\rm.
A segment $I{\subset}\s I_\beta$ is said to be $(p,\beta)$\it-equilateral\rm, if $\partial I{\subset}\Cal P$.
The $(p,\beta)$\it-rank \rm of a $(p,\beta)$-equilateral segment $I{=}[t,t']$
is defined to be the integer
$\s{rk}_{p,\beta}I{=}\s{rk}_{p,\beta}(t,t'){\leftrightharpoons}t'{-}t{+}p{-}|\gamma(t)$-$\gamma'(t')|$.
Sometimes we will omit the indices `$p$' and/or `$\beta$'.

\begin{prop}\label{prop3} If $I,J$ are $(p,\beta)$-equilateral segments and $I{\supset}J$, then
$\s{rk}I{\geqslant}\s{rk}J$.\end{prop}
\begin{proof} Let $I{=}[a,b]$, $J{=}[c,d]$ and thus $a{\leqslant}c{<}d{\leqslant}b$.
Let $m{\leftrightharpoons}\s{rk}J$, that is\hfil\penalty-10000
 $|\gamma'(d)$-$\gamma(c)|=d{-}c{+}p{-}m$.
By the $\triangle$-inequality,\hfil\penalty-10000
$|\gamma'(b)$-$\gamma(a)|\leqslant|\gamma'(b)$-$\gamma'(d)|{+}|\gamma'(d)$-$\gamma(c)|{+}|\gamma(c)$-$\gamma(a)|=$\hfil\penalty-10000
$\leqslant
(b{-}d){+}(d{-}c{+}p{-}m){+}(c{-}a)=b{-}a{+}p{-}m$.
It implies $\s{rk}I{\geqslant}m$.\end{proof}

\begin{prop}\label{prop4} $0\leqslant\s{rk}I\leqslant1{+}p$ for every $p$-equilateral segment $I$.\end{prop}
\begin{proof}
Let $I{=}[t,t']$. By the $\triangle$-inequality,\hfil\penalty-10000
$t'{-}t{+}p{-}\s{rk}(t,t')=|\gamma'(t')$-$\gamma(t)|\leqslant|\gamma'(t')$-$\gamma(t')|{+}|\gamma'(t')$-$\gamma(t)|=p{+}t'{-}t$,
and thus $\s{rk}(t,t'){\geqslant}0$.

On the other hand, since $|\gamma(0)$-$\gamma'(0)|\leqslant 1$, we have:\hfil\penalty-10000
 $t'{=}|\gamma'(0)$-$\gamma'(t')|{\leqslant}|\gamma'(0)$-$\gamma(0)|{+}|\gamma(0)$-$\gamma(t)|{+}|\gamma(t)$-$\gamma'(t')|{\leqslant}
1{+}t{+}t'{-}t{+}p{-}\s{rk}(t,t')$,
and thus
$\s{rk}(t,t'){\leqslant}1{+}p$.\end{proof}
\subsection{Rank decay lemma}
From now on we assume that, for the graph $S$, there exist numbers $Y{>}0$ and $\theta{\in}(0,1){\subset}\Bbb R$
such that, for every bigon $\alpha{\in}\s{Bg}S$ the following inequality holds (compare with the statement $\s A$ of Theorem 1 in Subsection 1.8):\hfil\penalty-10000
\vskip2pt
$\s{A}_{Y,\theta}(\alpha):$\quad$\#\s{w}_\alpha^{-1}(\Bbb R_{>Y}))<\theta{\cdot}|\alpha|$.\hfil\penalty-10000
\vskip2pt
Until the end of the proof of Theorem 1, such numbers $Y$ and $\theta$ will be fixed.\vskip3pt
\bf Definition 3\rm. For a fork $\beta$ in $S$, we call the elements of the set $\s w_\beta^{-1}(\Bbb R_{\leqslant 2Y+1})$
 \it small jumpers \rm of $\beta$.
We denote $\s{sJ}(\beta){\leftrightharpoons}\{$the small jumpers of $\beta\}$.

The following Proposition is a crucial point in the proof of Theorem 1.
The proof uses an idea of \cite[Lemma 1.5]{Pap} where the situation is rather
simpler.

\begin{prop} {\sc(Rank Decay Lemma)} Let $\beta{=}(\gamma,\gamma')$ be a fork
and $I{\subset}\s I_\beta$ a concatenation $J_-{\circ}J{\circ}J_+$ of
$(p,\beta)$-equilateral segments
$J_-,J,J_+$ for some $p{\in}\Bbb N$. If
$\s{rk}_{p,\beta}J_-{=}\s{rk}_{p,\beta}J_+{=}m$ and
 \begin{equation}\label{c3}
|J|\geqslant\theta{\cdot}|I|{+}(1{+}\theta)p,\end{equation}
then either $J{\cap}\s{sJ}(\beta)\ne\varnothing$ or $\s{rk}_{p,\beta}I\geqslant m{+}1$.\hfil\end{prop}
\begin{proof}
Let $t_0{<}t_1{<}t_2{<}t_3$ be $p$-jumpers of $\beta$,\begin{picture}(0,0)(-65,-15)
\put(-6,10){\makebox(0,0)[cr]{$\gamma$}}
\put(-5,10){\line(1,0){90}}
\put(-6,-10){\makebox(0,0)[cr]{$\gamma'$}}
\put(-5,-10){\line(1,0){90}}
\put(0,-12){\makebox(0,0)[tc]{$t_0$}}
\put(0,-10){\line(0,1){20}}
\put(10,1){\makebox(0,0)[bl]{$\sigma'$}}
\put(20,-10){\line(-1,1){20}}
\put(20,-12){\makebox(0,0)[tc]{$t_1$}}
\put(20,-10){\line(0,1){20}}
\put(35,-12){\makebox(0,0)[tc]{$a$}}
\put(45,12){\makebox(0,0)[bc]{$b$}}
\put(35,-10){\line(1,2){10}}
\put(35,-10){\line(0,1){20}}
\put(60,-12){\makebox(0,0)[tc]{$t_2$}}
\put(60,-10){\line(0,1){20}}
\put(80,-12){\makebox(0,0)[tc]{$t_3$}}
\put(80,-10){\line(0,1){20}}
\put(70,-2){\makebox(0,0)[tr]{$\sigma$}}
\put(80,-10){\line(-1,1){20}}
\end{picture}\hfil\penalty-10000
$J_-{=}[t_0,t_1]$, $J{=}[t_1,t_2]$,
 $J_+{=}[t_2,t_3]$.

Since $\s{rk}J_-=m$ one has:\hfil\penalty-10000
$|\gamma'(t_3)$-$\gamma(t_0)|\leqslant|\gamma'(t_1)$-$\gamma(t_0)|{+}|\gamma'(t_3)$-$\gamma'(t_1)|=(t_1{-}t_0{+}p
{-}m){+}(t_3{-}t_1)=$\hfil\penalty-10000
$=
t_3{-}t_0{+}p{-}m$.\hfil\penalty-10000
Similarly, since $\s{rk}J_+{=}m$ one has:\hfil\penalty-10000
$|\gamma'(t_3)$-$\gamma(t_0)|\leqslant|\gamma'(t_3)$-$\gamma(t_2)|{+}|\gamma(t_2)$-$\gamma(t_0)|=
(t_3{-}t_2{+}p{-}m){+}(t_2{-}t_0)=$\hfil\penalty-10000$
=t_3{-}t_0{+}p{-}m$.\hfil\penalty-10000
If one of these two inequalities is strict (then the other is also strict), then $\s{rk}I{=}\s{rk}[t_0,t_3]{>}m$.
Since the value of rank are integer, $\s{rk}I\geqslant m{+}1$ as claimed.

Suppose that both inequalities are equalities.
By joining each pair of $\{\gamma(t_0),\gamma'(t_1)\}$ and $\{\gamma(t_2),\gamma'(t_3)\}$
by a geodesic path in $S$ and adding to them the corresponding pieces of geodesics
$\gamma$, $\gamma'$, we obtain a (geodesic) bigon $(\sigma,\sigma')$ of length
$t_3{-}t_0{+}p{-}m$. See the picture where $\sigma$ is the side of the bigon
containing the point $\gamma(t_2)$, and $\sigma'$ is the side containing $\gamma'(t_1)$.

The length of the segment\hfil\penalty-10000
$L{\leftrightharpoons}\{y{\in}[0,|\sigma|]:\sigma(y){\in}\s{Im}\gamma,\sigma'(y){\in}\s{Im}\gamma'\}{=}[t_1{-}t_0{+}p{-}m,t_2{-}t_0]$,
is equal to\hfil\penalty-10000
 $|\sigma|{-}(t_1{-}t_0{+}p{-}m){-}(t_3{-}t_2{+}p{-}m){=}t_2{-}t_1{-}p{+}m$.\hfil\penalty-10000
It contains $|	L|{+}1=t_2{-}t_1{-}p{+}m{+}1$ integers.\hfil\penalty-10000
By (\ref{c3}) and Proposition \ref{prop4} one has:\hfil\penalty-10000
$\theta|\sigma|=\theta(t_3{-}t_0{+}p{-}m)\leqslant t_2{-}t_1{-}p{-}\theta m\leqslant|L|$.\hfil\penalty-10000
By $\s{A}_{Y,\theta}$, there exists $y{\in}L{\cap}\Bbb N$ such that\hfil\penalty-10000
 $Y\geqslant\s w_{\sigma,\sigma'}(y)=|\sigma(y)$-$\sigma'(y)|=|\gamma(b){-}\gamma'(a)|$, where\hfil\penalty-10000
$a{=}y{-}(t_1{-}t_0{+}p{-}m){+}t_1{=}t_0{+}y{-}p{+}m$, $b{=}t_0{+}y$.\hfil\penalty-10000
If $b<a$ then $b=a{+}(b{-}a)=a{+}(p{-}m)\geqslant a{+}({-}1)$ by Proposition \ref{prop4},
hence $b=a{-}1$ and\hfil\penalty-10000
 $\s w_\beta(a)=|\gamma(a)$-$\gamma'(a)|\leqslant|\gamma(a)$-$\gamma(a{-}1)|{+}|\gamma(b)$-$\gamma'(a)|\leqslant1{+}Y\leqslant 2Y{+}1$,\hfil\penalty-10000
so $a{\in}J{\cap}\s{sJ}(\beta)$ as claimed.\hfil\penalty-10000
If $b\geqslant a$ then, using the assumption that $\beta$ is a fork, we have:\hfil\penalty-10000
$b=|\gamma(b)$-$\gamma(0)|\leqslant
|\gamma(b)$-$\gamma'(a)|{+}|\gamma'(a)$-$\gamma'(0)|{+}|\gamma'(0)$-$\gamma(0)|
\leqslant Y{+}a{+}1$,
and thus\hfil\penalty-10000
 $|b{-}a|=b{-}a\leqslant Y{+}1$,
$\s{w}_\beta(a)\leqslant|\gamma(a){-}\gamma(b)|{+}|\gamma(b)$-$\gamma'(a)|\leqslant 2Y{+}1$,
which means that $a$ is a small jumper: $a{\in}J{\cap}\s{sJ}(\beta)$.\end{proof}
\vskip3pt
\section{Regular segments and small jumpers}
\subsection{Dense value of the width function}
In addition to the assumptions on $S$  already made (including $\s A_{Y,\theta}$), we assume
that there exist numbers
$Z{\in}\Bbb N_{>0}$ and $\lambda{\in}(0,1){\subset}\Bbb R$
such that, for every bigon
$\alpha{\in}\s{Bg}S$ the following inequality holds (see the statement $\s B$ of Theorem 1 in Subsection 1.8):\hfil\penalty-10000
$\s B_{Z,\lambda}:$\quad$\#\s{w}_\alpha^{-1}(\Bbb R_{>Z})<\Frac\lambda{4Y\overset{\vphantom.}{+}2}$.\hfil\penalty-10000
Until the end of the proof of Theorem 1, such numbers $Z$ and $\lambda$ will be fixed.\vskip3pt

We also assume that a fixed number $\nu{\in}(\lambda,1){\subset}\Bbb R$ (as in the Dense Value Lemma) is chosen.
We are going to apply the Dense Value Lemma for the numbers
$a_-=a_+=2Y{+}1$, $Z$ from $\s B_{Z,\lambda}$ and $\varepsilon=\Frac\lambda{4Y\overset{\vphantom.}{+}2}$,
$\nu{\in}(\lambda,1)$.

It follows from the Dense Value Lemma that
there exists $R\in\Bbb N_{>0}$ such that,
if the length $l$ of a band $\beta{=}(\alpha,\alpha')$ is greater then $R$ and
$\s w_\beta(0)$ and $\s w_\beta(l)$ do not exceed $2Y{+}1$, then
there exists $p{\in}\Bbb N_{\leqslant D}$, where $D=aZ{+}a_-=(2Y{+}1)(2Z{+}1)$,
such that $\#\s w_\beta^{-1}p\geqslant\varrho{\cdot}l$, where $\varrho\leftrightharpoons\Frac{1{-}\nu}{D{+}1}$.

Until the end of the proof of Theorem 1 the symbols $D$ and $\varrho$ denote the numbers
just defined.
\subsection{Regular segments for forks}
We will use the integer part function $\Bbb R_{\geqslant0}\to\Bbb N$:
for $r{\in}\Bbb R_{\geqslant 0}$, let $r_{\Bbb N}\leftrightharpoons\s{max}(\Bbb N_{\leqslant r})$.

Keeping the notations define recurrently a sequence of numbers $\{n_i:i{\in}\Bbb N\}$ as follows:
\begin{equation}\label{e4}
n_0\leftrightharpoons6,\quad
 n_{i+1}\leftrightharpoons1{+}\mkern-4mu\left(\Frac{2n_i{+}1}{1{\overset{\vphantom.}-}\theta}\right)_{\mkern-5mu\Bbb N}.\end{equation}
Further, until the end of the proof of Theorem 1, let
\begin{equation}\label{e5}
N\leftrightharpoons n_{D+1},\quad \mu\leftrightharpoons\Frac N{N\overset{\vphantom.}{-}1},\quad
K\leftrightharpoons1{+}(\s{log}_\mu(2/\varrho))_{\Bbb N}.\end{equation}

\textbf{Definition} ($p$-density and $p$-regularity).
For a fork $\beta$, a segment $I{\subset}\s I_\beta$, and a value $p{\in}\Bbb N$,
the number
 $\s d_{\beta,p}I\leftrightharpoons\Frac{\#(I{\cap}\s{w}_\beta^{-1}p))}{|I|}$
is called the $p$\it-density \rm of $I$ with respect to $\beta$.

A value $p$ is called \it dense value \rm(of the width function $\s w_\beta$ on $I$),
if the $p$-density of $I$ with respect to $\beta$ is at least $\varrho$, that is
 $\#\left(I{\cap}\s{w}_\beta^{-1}p\right)\geqslant\varrho{\cdot}|I|$ where $\varrho=\Frac{1-\nu}{D+1}$ (see the Dense Value Lemma).

For a dense value $p$ of $\s w_\beta$ on $I$,
we call an arbitrary segment $J{\subset}I$ (whose endpoints are not necessarily integers)
a $(\beta,p)$\it-regular\rm, if, after the subdivision of $J$ into $N$ \textbf{pairwise congruent}
subsegments (called \textit{pieces of} $J$) each such piece contains at least
one $(\beta,p)$-jumper.

\subsection{Existence of regular segments}
\begin{prop}\label{p6} Let $\beta,I,p,\varrho$ as above.
If $|I|\geqslant N^K$, $p\leqslant D$ then there exists a $(\beta,p)$-regular segment $J{\subset}I$
of length ${\geqslant}\Frac{|I|}{N^K}$.\end{prop}
\begin{proof}
Subdivide the segment $I$ in $N$ pieces.
If $I$ is not $(\beta,p)$-regular then at least one of the pieces does not intersect $\s w_\beta^{-1}p$
then the $p$-density of at least one of the other pieces is ${\geqslant}\mu\varrho$
since otherwise
 $\#(I{\cap}\s{w}_\beta^{-1}p)<(N{-}1){\cdot}\Frac{|I|}N{\cdot}\mu\varrho=\varrho|I|$
contradicting the definition of $p$-dense value.
Denote this piece of density ${\geqslant}\mu\varrho$ by $I_1$.

If $I_1$ is not $p$-regular then we repeat the argument: subdividing $I_1$ into
$N$ pieces, we find a piece $I_2$ of $I_1$ with $p$-density ${\geqslant}\mu^2\varrho$.

This process of ``subdivision with increasing $p$-density'' cannot contain more than
$\s{log}_\mu(2/\varrho)$ steps, since, for $k{\leqslant}K$, one has\hfil\penalty-10000
 $|I_k|=|I|/N^k\geqslant1$
and $\mu^k\varrho\leqslant\s d_{\beta,p}I_k=\Frac{\#(I_k{\cap}\s{w}_\beta^{-1}p)}{|I_k|}\leqslant
\Frac{|I_k|{+}1}{|I_k|}\leqslant2$,
so $k\leqslant\s{log}_\mu(2/\varrho)$.\end{proof}
\subsection{Existence of small jumper inside a regular segment}
We keep the notation for the constants $Y,Z,\theta,\varrho,D,N$ etc.
\begin{prop}\label{p7} For every fork $\beta$, the interior of every $(\beta,p)$-regular segment $I$
 of length $\geqslant\Frac{DN(1{+}\theta)}{1{\overset{\vphantom.}-}\theta}$
intersects $\s{sJ}(\beta)$.\end{prop}
\begin{proof}
Let $\Cal P\leftrightharpoons I{\cap}\s w_\beta^{-1}p$.
Recall that a subsegment $J$ of $I$ with $\partial J{\subset}\Cal P$ is $(p,\beta)$\textit{-equilateral}.\par
Let us say that a subsegment $J$ of $I$ is \textit{entire} if it is a concatenation of pieces of $I$ (see the definition in Subsection 4.2).
The number of pieces of $I$ in a subsegment $J$ entire in $I$ we call $I$\textit{-reduced length}
of $J$. We denote the $I$-reduced length of $J$ by $\s{rl}_IJ$.

In particular $\s{rl}_II{=}N$.

The statement of Proposition \ref{p7} is a particular case (with $k{=}p{+}1$)
of the following
\begin{prop}\label{p8}
Let $I$ be a $(\beta,p)$-regular segment of length ${\geqslant}\Frac{DN(1{+}\theta)}{1{\overset{\vphantom.}-}\theta}$,
let $k{\in}\Bbb N_{\leqslant K}$ and
$J$ be an entire subsegment of $I$ having $I$-reduced length
$n_k$, where $n_k$ satisfies (\ref{e4}).
Let $H$ be the maximal $(p,\beta)$-equilateral subsegment of $J$.
If $\s{rk}_{p,\beta}H\leqslant k$ then
there exists a small jumper in $H{\setminus}\partial J$.\end{prop}
\textit{Proof.} We proceed by induction on $k$.\par For $k{=}0$ we have $\s{rl}_IJ{=}n_0{=}6$.
Using the regularity of $I$ choose one $p$-jumper in the \textbf{second} $I$-piece of $J$
 and one $p$-jumper in the \textbf{fifth} $I$-piece of $J$.
Denote by $H'$ the $p$-equilateral segment whose endpoints are the choosen $p$-jumpers.
So $H'{\cap}\partial J{=}\varnothing$ and it suffices to find a small jumper in $H'$.

We have: $|H'|\geqslant2{\cdot}\Frac{|J|}6
\geqslant\Frac{2D(1{+}\theta)}{1{\overset{\vphantom.}-}\theta}\geqslant
e{\leftrightharpoons}\Frac{p(1{+}\theta)}{1{\overset{\vphantom.}-}\theta}$ (see 4.1).
Hence\hfil\penalty-10000
$|H'|(1{-}\theta)\geqslant p(1{+}\theta)$, $|H'|\geqslant p(1{+}\theta){+}\theta|H'|$,
$|H'|{-}p\geqslant\theta(|H'|{+}p)$.

Let $H'{=}[t_0,t_1]$.
By joining the points $\gamma(t_\iota)$ and $\gamma'(t_\iota)$
by geodesic segments we obtain, by `$\s{rk}_pH'{=}0$' and Proposition \ref{prop3},
a bigon $\sigma{=}(\sigma_0,\sigma_1)$ of length $|H'|{+}p$,
where $\sigma_0$ is the concatenation $[\gamma(t_0),\gamma(t_1)]{\circ}[\gamma(t_1),\gamma'(t_1)]$
and $\sigma_1$ is the
 concatenation $[\gamma(t_0),\gamma'(t_0)]{\circ}[\gamma'(t_0),\gamma'(t_1)]$.

This situation resembles that of the proof of Rank Decay
 Lemma, and we repeat the argument thereof.

The length of the segment
$L{\leftrightharpoons}\{y{\in}\left[0,|\sigma|\right]:\sigma(y){\in}\s{Im}\gamma,\sigma'(y){\in}\s{Im}(\gamma')\}$
is equal to
 $|\sigma|{-}2p=|H'|{-}p\geqslant\theta(|H'|{+}p)=\theta|\sigma|$.\begin{picture}(0,0)(-98,30)
\put(-6,15){\makebox(0,0)[cr]{$\gamma$}}
\put(-5,15){\line(1,0){40}}
\put(-6,-15){\makebox(0,0)[cr]{$\gamma'$}}
\put(-5,-15){\line(1,0){40}}
\put(1,-17){\makebox(0,0)[tc]{$t_0$}}
\put(10,-17){\makebox(0,0)[tc]{$a$}}
\put(26,16.5){\makebox(0,0)[bc]{$b$}}
\put(0,-15){\line(0,1){30}}
\put(0,0){\makebox(0,0)[cr]{$\sigma_1$}}
\put(30,-17){\makebox(0,0)[tc]{$t_1$}}
\put(30,-15){\line(0,1){30}}
\put(10,-15){\line(1,2){15}}
\put(10,-15){\line(0,1){30}}
\put(31,0){\makebox(0,0)[cl]{$\sigma_0$}}
\end{picture}\par

Since $|H'|{-}p\geqslant\Frac{2D(1{+}\theta)}{1{\overset{\vphantom.}-}\theta}{-}D{\geqslant}D{\geqslant}1$
we have $\Bbb N{\cap}L{\ne}\varnothing$.

By $\s A_{Y,\theta}$, there exists $y{\in}\Bbb N{\cap}L$ such that $\s w_\sigma(y){\leqslant}Y$.
We have\hfil\penalty-10000
 $Y\geqslant|\sigma(y)$-$\sigma'(y)|=|\gamma(b)$-$\gamma'(a)|$ where $a{=}t_0{+}y{-}p$,
$b{=}t_0{+}y$.\hfill\break Since $\beta$ is a fork, we have\hfill\break
$b{=}|\gamma(b)$-$\gamma(0)|\leqslant|\gamma(b)$-$\gamma'(a)|+|\gamma'(a)$-$\gamma'(0)|+|\gamma'(0)$-$\gamma(0)|\leqslant
Y{+}a{+}1$,\hfill\break $|b{-}a|{=}b{-}a\leqslant Y{+}1$,
 $\s w_\beta(a)\leqslant|\gamma(a)$-$\gamma(b)|{+}|\gamma(b)$-$\gamma'(a)|{\leqslant}2Y{+}1$.

The small jumper $a{\in}H'$ is different from the endpoints of $J$ by the choice of $H'$.
\vskip3pt

Suppose that $k{>}0$ and that, for the values less than $k$, the assertion of Proposition \ref{p8}
is true.

Let $J$ be an $I$-entire subsegment of $I$ and $\s{rl}_IJ{=}n_k$.
Express $J$ as a concatenation $J_-{\circ}J_0{\circ}J_+$ of $I$-entire subsegments so that
$\s{rl}_IJ_-{=}\s{rl}_IJ_+{=}n_{k-1}$.
Let $H_-,H_+$ denote the maximal $(p,\beta)$-equilateral subsegments of $J_-$ and $J_+$ respectively.
By Proposition \ref{prop3}, $\s{rk}_{p,\beta}H_\pm{\leqslant}k$.
If the $(p,\beta)$-rank of at least one of the segments $H_\pm$ is less than $k$ then
it contains a small jumper different from each of the endpoints of $J$
by the hypothesis of induction.

If $\s{rk}_pH_-=\s{rk}_pH_+=k$, then we will verify the hypothesis of the Rank Decay Lemma.

Let $H_-{=}[t_0,t_1]$, $H_+{=}[t_2,t_3]$.
By the choice of $H_\pm$, $t_2{-}t_1\geqslant e(n_k{-}2n_{k-1})$.
It follows from $(\ref{e4})$ that
$n_k>\Frac{2n_{k-1}{+}1}{1{\overset{\vphantom.}-}\theta}$,
 $(1{-}\theta)n_k>2n_{k-1}{+}1$,\hfil\penalty-10000
 $n_k{-}2n_{k-1}>\theta n_k{+}1$, thus\hfil\penalty-10000
$t_2{-}t_1>e(\theta n_k{+}1)=e\theta n_k{+}e=
\theta|J|{+}e=\theta|J|{+}\Frac{p(1{+}\theta)}{1{\overset{\vphantom.}-}\theta}>\theta|J|{+}(1{+}\theta)p$,
which is required in Rank Decay Lemma.
The segment $[t_1,t_2]$ contains a small jumper that can not be
an endpoint of $J$.

Propositions \ref{p7} and \ref{p8} are proved.\end{proof}
\section{Proof of Theorem 1}
Using the assumptions $\s A$ and $\s B$ we find numbers  $\theta,\lambda\in(0,1)\subset\Bbb R$
such that
every bigon $\beta{\in}\s{Bg}S$ satisfies the conditions:
\vskip5pt
\noindent $\s A_{Y,\theta}(\beta):$\quad$\#\s{w}_\beta^{-1}(\Bbb R_{>Y})<\theta{\cdot}|\beta|$\ \ and\hfil\penalty-10000
$\s B_{Z,\lambda}(\beta):$\quad$\#\s{w}_\beta^{-1}(\Bbb R_{>Z})<\Frac\lambda{4Y\overset{\vphantom.}{+}2}{\cdot}|\beta|$.

\vskip5pt
We fix these constants keeping the previous notations.

For a bigon $\beta{\in}\s{Bg}S$ the endpoints $0$ and $|\beta|$ belong to $\s{sJ}{=}\{$the small jumpers of $\beta\}$.
Thus, to conclude that the graph $S$ is hyperbolic it suffuces, by $(*)$ of 1.5,
to find a constant $C$ such that
for every bigon $\beta{\in}\s{Bg}S$ the distance between every two \textbf{neighbors} (in the natural ordering)
in $\s{sJ}(\beta)$ is less than or equal to $C$.

We will show that one can take
 $C\leftrightharpoons\s{max}\left\{R,\Frac{DN^{K+1}(1{+}\theta)}{1{\overset{\vphantom.}-}\theta}\right\}$.\par

Let $I\subset\s I_\beta$, $\partial I\subset\s{sJ}(\beta)$, $|I|\geqslant C$.
Since $|I|\geqslant R$, by the Dense Value Lemma, there exists a dense value $p$ of the function $\s w_\beta|_I$.
Since
 $|I|\geqslant\Frac{DN^{K+1}(1{+}\theta)}{1{\overset{\vphantom.}-}\theta}$,
by Proposition \ref{p6}, there exists a $(\beta,p)$-regular subsegment $J\subset I$
of length $\geqslant\Frac{|I|}{N^K}\geqslant\Frac{DN(1{+}\theta)}{1{\overset{\vphantom.}-}\theta}$.
By Proposition \ref{p7} there exists $t{\in}\s{sJ}(\beta){\setminus}\partial J$,
and so $t{\in}I{\setminus}\partial I$. Hence the points of $\partial I$ can not be \textbf{neighbors}
in $\s{sJ}(\beta)$.

Theorem 1 is proved.
\vskip3pt
Since $\s{TAS}$ implies both conditions $\s A$ and $\s B$ of Theorem 1, we obtain

\sc Corolary \sl Every graph satisfying $\s{ TAS}$ is hyperbolic\rm.
\vskip3pt
\section{Van Kampen area of geodesic bigons}
We will deduce $\s{TAS}$ from the boundness of the set (\ref{vk}) (Subsection 1.7).
\subsection{$\square$-inequality}
Recall that the Van Kampen area of a closed path $\gamma$ in the Cayley graph related to
a finite presentation $\langle X;\Cal R\rangle$ is the smallest $n{\in}\Bbb N$ such that
the element $\widetilde\gamma$ of the free group $F=\langle X\rangle$
corresponding to $\gamma$ is expressible in $F$ in the form
$\widetilde\gamma=\prod_{i=1}^ng_i^{-1}r_ig_i$, where $r_i{\in}\Cal R$ (the set $\Cal R{\subset}F$ is supposed to
be symmetric).

It is well-kown (see for example \cite[Proposition III.9.2]{LSh}), that the Van Kampen area
is equal to the minimal number of cells in the Van Kampen diagrams with the boundary label $\widetilde\gamma$.

We will use the following property of the Van Kampen area proved by Bowditch \cite[Proposition 5.7., p 108]{Bo},
which we call the $\square$\it-inequality\rm.
Originally it was formulated for an arbitrary path-metris space $S$
and a function $\s a:\s{rL}(S){\leftrightharpoons}\{$the rectifiable closed paths $\Bbb S^1\to S\}\to\Bbb R_{\geqslant0}$ where $\Bbb S^1$ is
the standard circle $\{z{\in}\Bbb C:|z|{=}1\}$.
\vskip3pt
$(\square)$\sl There exists $\omega{\in}\Bbb R_{>0}$
such that for every rectifiable\begin{picture}(0,0)(-60,25)
\put(-21,21){\makebox(0,0)[rb]{$P_0$}}
\put(21,21){\makebox(0,0)[lb]{$P_1$}}
\put(20,-18){\makebox(0,0)[lt]{$P_2$}}
\put(-20,-18){\makebox(0,0)[rt]{$P_3$}}
\put(-20,20){\line(1,0){40}}
\put(-20,20){\line(0,-1){40}}
\put(20,-20){\line(-1,0){40}}
\put(20,-20){\line(0,1){40}}
\end{picture}\hfil\penalty-10000
 closed curve
$\Bbb S^1\overset\psi\to S$
and every subdivision
 of $\Bbb S^1$ into four\hfil\penalty-10000
 segments
 $[P_0,P_1]$, $[P_1,P_2]$, $[P_2,P_3]$,$[P_3,P_0]$,
then
$\s{a}(\psi){\geqslant}\omega{\cdot}d_0{\cdot}d_1$,\hfil\penalty-10000
 where
 $d_0{=}\s{dist}(\psi[P_0,P_1],\psi[P_2,P_3])$,\hfil\penalty-10000
 $d_1{=}\s{dist}(\psi[P_1,P_2],\psi[P_3,P_0])$\rm.
 \vskip3pt
\sc Remark\rm. In the cited paper \cite{Bo} Bowditch gave an axiomatic definition
of an ``area-like'' function $\s a$ on $\s{rL}(S)$ where
$\square$-inequality is one of the axioms. The other axiom
is an analog of the $\triangle$-inequality.
It is proved in \cite{Bo} that the boundness of
$\left\{\Frac{\s a(\gamma)}{\s{length}(\gamma)}:\gamma{\in}\s{rL}(S)\right\}$,
where $\s a$ is a function satisfying the $\triangle$-inequality and the $\square$-inequality for any closed rectifiable $\gamma$, implies the hyperbolicity of the space $S$.

Bowditch also proved that the Van Kampen area for any finite presentation satisfies the axioms.

Clearly a ``discrete'' version of $\square$-inequality also holds for every finite presentation $\Cal P$
(where the constant $\omega$ depends on $\Cal P$).

In our case we assume the boundness of the smaller set (\ref{vk}) where $\gamma$ is the union of two geodesic segments (see 1.7).

The Bowditch area axioms are not sufficient for our proof.

It is an interesting question whether a version of Theorem 2 where the Van Kampen
area is replaced by a
Bowditch area function is still true.
\subsection{Special pieces of a Van Kampen diagram for a bigon}
We will prove that if a graph $S$ does not satisfy $\s{TAS}$ then the set (\ref{vk}) is not bounded.

Assume that the $\s{TAS}$ condition does not hold.
Thus there exists $\varepsilon{>}0$ such that
\begin{equation}\label{noTAS}
\forall d\in\Bbb R\ \exists\beta\in\s{Bg}S:
\#\s w_\beta^{-1}(\Bbb N_{>d})>\varepsilon{\cdot}|\beta|.\end{equation}
We fix such $\varepsilon$ until the end of the proof of Theorem 2.

Let $d$ be a positive real number and $\beta{=}\beta_d{=}(\gamma,\gamma')$ be a bigon with\hfil\penalty-10000
 $\#\s w_\beta^{-1}(\Bbb N_{>d})>\varepsilon{\cdot}|\beta|$.
We will show that $\Frac{\s{a}(\beta)}{|\beta|}>\epsilon d$ for some $\epsilon$ depending only
on $\varepsilon$ and $\Cal P$.

It is convenient to extend the width function $\s w_\beta$ to
a \textbf{contiunous} function defined on the whole segment $\s I_\beta{\subset}\Bbb R$.

Recall that a Van Kampen diagram for a symmetric
presentation $\mathcal P{=}\langle X;\mathcal R\rangle$ is finite two-dimensional $\s{CW}$-complex $W$
homeomorphic to a singular disk contained in the sphere $\Bbb S^2$ together with
a simplicial (labelling) map $W^1\overset\kappa\to S$ that
maps edges to edges such that the boundary label of each 2-cell is a defining relation ${\in}\mathcal R$
\cite[Section III.9]{LSh}.

Consider a Van Kampen diagram with the boundary label $\gamma^{-1}{\circ}\gamma'$.
Since $\gamma$ and $\gamma'$ are geodesic words of equal length,
the only possible singularity
on $\partial W$ is the coincidence of vertices $\gamma(x)$ and $\gamma'(x)$ for some values $x{\in}\s I_\beta$,
 as indicated on the picture below.\par
\begin{picture}(0,85)(-1,-85)
\put(90,-13){\makebox(0,0)[cc]{$\gamma(x_1)$}}
\put(90,-18){\vector(0,-1){22}}
\put(90,-71){\makebox(0,0)[cc]{$\gamma'(x_1)$}}
\put(90,-66){\vector(0,1){19}}
\put(241,-13){\makebox(0,0)[cc]{$\gamma(x_2)$}}
\put(241,-19){\vector(0,-1){22}}
\put(241,-71){\makebox(0,0)[cc]{$\gamma'(x_2)$}}
\put(241,-66){\vector(0,1){19}}
\put(-14,-70){
\pdfximage{\mpPath 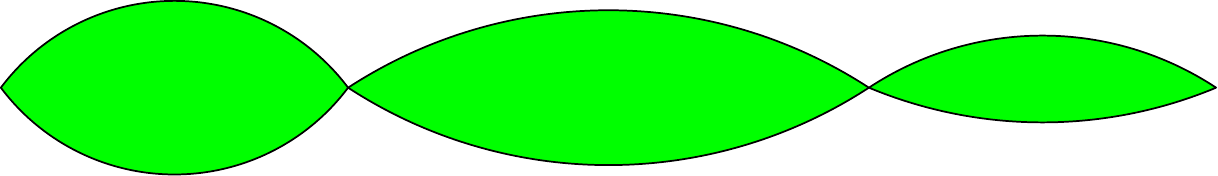}\pdfrefximage\pdflastximage}
\end{picture}\vskip65pt
If the inequality `$\#\s w_\beta^{-1}(\Bbb N_{>d})>\varepsilon{\cdot}|\beta|$' holds for $\beta$ then
it holds for the closure of one of the connected
components of $W{\setminus}\{$the singular points\}. Thus we can assume that the path $\gamma^{-1}{\circ}\gamma'$
is a \bf simple loop \rm and the complex $W$ is \bf homeomorpic \rm to the disk $\Bbb D^2$.

On the \bf graph \rm $W^1$ (the one-dimensional skeleton of $W$)
we fix the simplicial metric with the length of each edige equal 1.
We denote the corresponding distance between points $p,q{\in}W^1$ by
$|p$-$q|_W$. This is the minimal length of paths in $W^1$ from $p$ to $q$.

The map $\kappa$ sends edges to edges, so it does not increase the distance.
Since the segments $\gamma$ and $\gamma'$ are geodesic in $S$, they are also geodesic in $W^1$.
Hence without loss of generality we can assume that $\kappa$ is the identity map on the images of $\gamma$ and $\gamma'$.

For $x{\in}\s I_\beta$, denote by $\widetilde{\s w}_\beta(x)$
the distance $|\gamma(x)$-$\gamma'(x)|_W$.
So, we have\hfil\penalty-10000
 $\widetilde{\s w}_\beta(x){\geqslant}\s w_\beta(x)$ for each $x{\in}\s I_\beta$.
We also have $\s w_\beta^{-1}(\Bbb R_{>d}){\subset}\widetilde{\s w}_\beta^{-1}(\Bbb R_{>d})$. Hence
$\#\widetilde{\s w}_\beta^{-1}(\Bbb N_{>d})>\varepsilon{\cdot}|\beta|$.

For $x{\in}\widetilde{\s w}_\beta^{-1}(\Bbb R_{>d})$ denote by $J_x(\beta,d)$ the maximal \textbf{open} segment
 ${\subset}\s I_\beta$ such that every its point $t$ satisfies the inequality
$\widetilde{\s w}_\beta(y)>{1\over5}d$. Then, by continuity of $\widetilde{\s w}_\beta$, for the boundary points
$y{\in}\partial J_x(\beta,d)$ we have the equality\hfil\penalty-10000
$\widetilde{\s w}_\beta(y)={1\over5}d$.

\begin{prop}\label{p9} $|J_x(\beta,d)|>{4\over5}d$.\end{prop}
\begin{proof} Let $J_x(\beta,d){=}(x_-,x_+)$, so that $x_-{<}x{<}x_+$.
It suffices to show that both numbers $x{-}x_-$ and $x_+{-}x$ are greater than
 ${2\over5}d$.\hfil\penalty-10000
The proof is the same for both numbers.
So we only give it for $x{-}x_-$.\par
If there were $x{-}x_-\leqslant{2\over5}d$, then, by $\triangle$-inequality we would have:\hfil\penalty-10000
 $\widetilde{\s w}_\beta(x)=|\gamma(x)$-$\gamma'(x)|_W\leqslant
|\gamma(x)$-$\gamma(x_-)|{+}|\gamma(x_-)$-$\gamma'(x_-)|_W{+}|\gamma'(x_-)$-$\gamma'(x)|=$\hfil\penalty-10000
$=
2(x{-}x_-){+}\widetilde{\s w}_\beta(x_-)\leqslant2{\cdot}{2\over5}d{+}{1\over5}d=d$,\hfil\penalty-10000
contradicting the assumption `$\widetilde{\s w}_\beta(x)>d$'.\end{proof}

Let $\Cal A_d(\beta){\leftrightharpoons}\{$the connected components of the open set\hfil\penalty-10000
 ${\cup}\{J_x(\beta,d):x{\in}\widetilde{\s w}_\beta^{-1}(\Bbb R_{>d})\}\subset\s I_\beta$.\hfil\penalty-10000
By Proposition \ref{p9} the length of each open segment $I{\in}\Cal A_d(\beta)$ is greater than ${4\over5}d$.
The segments in $\Cal A_d$ are pairwise disjoint and the their union contains
$\widetilde{\s w}_\beta^{-1}(\Bbb R_{>d})$. Hence\begin{equation}\label{sumLength}
\sum\{|I|:I{\in}\Cal A_d(\beta)\}\geqslant\varepsilon{\cdot}|\beta|.\end{equation}

Consider a fixed segment $I=(x_-,x_+){\in}\Cal A_d(\beta)$.
By joining the pairs
$\{\gamma(x_-),\gamma'(x_-)\}$
and $\{\gamma(x_+),\gamma'(x_+)\}$ by geodesic segments in the graph $W^1$,\hfil\penalty-10000
we obtain a geodesic quadrilateral $Q_I$ in $W$, whose one pair of the opposite sides is
a pair of pieces of $\gamma$ and $\gamma'$.

Denote the sides of $Q_I$ by $[\gamma(x_-),\gamma'(x_-)]$, $[\gamma(x_+),\gamma'(x_+)]$, $[\gamma(x_-),\gamma(x_+)]$,
$[\gamma'(x_-),\gamma'(x_+)]$.\hfil\penalty-10000
In order to apply the $\square$-inequality, let us estimate the $W^1$-distance between the
opposite sides of $Q_I$.

\begin{prop} $\s{dist}_W([\gamma(x_-),\gamma(x_+)],[\gamma'(x_-),\gamma'(x_+)])>{1\over10}d$.\end{prop}
\begin{proof}
Assuming the converse, consider
 $x,x'{\in}(x_-,x_+)$,
such that\hfil\penalty-10000
$|\gamma(x)$-$\gamma'(x')|_W{\leqslant}{1\over10}d$.\hfil\penalty-10000
Without loss of generality we can assume that\begin{picture}(0,0)(-87,39)
\put(-40,30){\circle*{3}}
\put(40,30){\circle*{3}}
\put(-38,33){\makebox(0,0)[cb]{$\gamma'(x_-)$}}
\put(40,33){\makebox(0,0)[cb]{$\gamma'(x_+)$}}
\put(-45,30){\line(1,0){90}}
\put(10,28){\makebox(0,0)[ct]{$\gamma'(x')$}}
\put(10,30){\circle*{3}}
\put(-40,30){\line(0,-1){60}}
\put(40,30){\line(0,-1){60}}
\put(-45,-30){\line(1,0){90}}
\put(-38,-33){\makebox(0,0)[ct]{$\gamma(x_-)$}}
\put(-40,-30){\circle*{3}}
\put(-20,-28){\makebox(0,0)[cb]{$\gamma(x)$}}
\put(-20,-30){\circle*{3}}
\put(38,-33){\makebox(0,0)[ct]{$\gamma(x_+)$}}
\put(40,-30){\circle*{3}}
\end{picture}\hfil\penalty-10000
 $x\leqslant x'$.
Since $\gamma(0){=}\gamma'(0)$, the $\triangle$-inequality gives:\hfil\penalty-10000
$x'{=}|\gamma'(0)$-$\gamma'(x')|{\leqslant}|\gamma(0)$-$\gamma(x)|{+}|\gamma(x)$-$\gamma'(x')|_W{\leqslant}x{+}{1\over10}d$.\hfil\penalty-10000
So $x'{-}x{\leqslant}{1\over10}d$. For $y{\leftrightharpoons}\Frac{x{+}x'}2$ we have:\hfil\penalty-10000
$y{-}x={1\over2}(x'{-}x)\leqslant{1\over20}d$,
 $x'{-}y\leqslant{1\over20}d$.\hfil\penalty-10000
By the $\triangle$-inequality,\hfil\penalty-10000
${1\over5}d<\widetilde{\s w}_\beta(y)=|\gamma(y)$-$\gamma'(y)|_W\leqslant$\hfil\penalty-10000
$\leqslant
|\gamma(y)$-$\gamma(x)|{+}|\gamma(x)$-$\gamma'(x')|_W{+}|\gamma'(x')$-$\gamma'(y)|\leqslant{1\over20}d{+}{1\over10}d{+}{1\over20}d={1\over5}d$ which is a conradiction.\end{proof}

\begin{prop} $\s{dist}_W([\gamma(x_-),\gamma'(x_-)],[\gamma(x_+),\gamma'(x_+)])>{1\over2}|I|$.\end{prop}
\begin{proof} Assuming the converse, consider a pair of\begin{picture}(0,0)(-77,40)
\put(-40,30){\circle*{3}}
\put(40,30){\circle*{3}}
\put(-38,33){\makebox(0,0)[cb]{$\gamma'(x_-)$}}
\put(40,33){\makebox(0,0)[cb]{$\gamma'(x_+)$}}
\put(-45,30){\line(1,0){90}}
\put(-37,10){\makebox(0,0)[lc]{$y_-$}}
\put(-40,10){\circle*{3}}
\put(-40,30){\line(0,-1){60}}
\put(40,30){\line(0,-1){60}}
\put(-45,-30){\line(1,0){90}}
\put(-38,-33){\makebox(0,0)[ct]{$\gamma(x_-)$}}
\put(-40,-30){\circle*{3}}
\put(38,-5){\makebox(0,0)[cr]{$y_+$}}
\put(40,-5){\circle*{3}}
\put(38,-33){\makebox(0,0)[ct]{$\gamma(x_+)$}}
\put(40,-30){\circle*{3}}
\end{picture}\hfil\penalty-10000
points
$y_-{\in}[\gamma(x_-),\gamma'(x_-)]$, $y_+{\in}[\gamma(x_+),\gamma'(x_+)]$,\hfil\penalty-10000
such that
 $|y_-$-$y_+|_W\leqslant{1\over2}|I|=\Frac{x_+{-}x_-}2$.

By the $\triangle$-inequality,\hfil\penalty-10000
 $|I|{=}x_+{-}x_-{=}|\gamma(x_-)$-$\gamma(x_+)|\leqslant$\hfil\penalty-10000
$\leqslant|\gamma(x_-)$-$y_-|_W{+}|y_-$-$y_+|_W{+}|y_+$-$\gamma(x_+)|_W\leqslant$\hfil\penalty-10000
${1\over5}d{+}{1\over2}|I|{+}{1\over5}d$,
since\hfil\penalty-10000
 $|\gamma(x_-)$-$y_-|_W\leqslant|\gamma(x_-)$-$\gamma'(x_-)|_W={1\over5}d$, and,\hfil\penalty-10000
similarly,
$|\gamma(x_+)$-$y_+|_W\leqslant{1\over5}d$.
Thus, $|I|\leqslant{4\over5}d$,
 contradicting Proposition \ref{p9}.\end{proof}

By the $\square$-inequality we have:\hfil\penalty-10000
\sc Corollary\sl. For every $I{\in}\Cal A_d(\beta)$ one has
$\s{a}Q_I>\omega{\cdot}{1\over10}d{\cdot}{1\over2}|I|$ where $\omega$ depends only on \rm$\Cal P$.

\subsection{A lemma about segments}
For a finite set $\Cal A$ of bounded segments${\subset}\Bbb R$, denote
$\|\Cal A\|{\leftrightharpoons}\sum_{I\in\Cal A}|I|$.

\textbf{Lemma.}\textit{
Let $\Cal A$ be a finite set of open bounded mutually disjoint segments
 ${\subset}\Bbb R$ of length ${>}a{>}0$. There exists a set $\Cal{B{\subset}A}$ such that
$\|\Cal B\|{\geqslant}{1\over3}\|\Cal A\|$ and the distance between any two segments ${\in}\Cal B$
is greater than $a$.}
\begin{proof}
Induction on $m{\leftrightharpoons}\#\Cal A$.
If $m{\leqslant}1$ then, trivally, one takes $\Cal{B{=}A}$.

Suppose that $m{\geqslant}2$ and that, for the smaller values of $m$, the assertion is true.

Let $I$ be a segment${\in}\Cal A$ of maximal length and let\hfil\penalty-10000
 $\Cal D{\leftrightharpoons}\{D{\in}\Cal A{\setminus}\{I\}:\s{dist}(I,D){\leqslant}a\}$.
It follows from the hypothesis that the cardinality of
$\Cal D$ is at most 2.
Since the length of $I$ is maximal
$|I|{\geqslant}{1\over3}\|(\Cal D{\cup}\{I\})\|$.
By the induction assumption, the set $\Cal A'{\leftrightharpoons}\Cal A{\setminus}((\Cal D{\cup}\{I\}))$
contains a subset $\Cal B'$ such that $\|\Cal B'\|{\geqslant}{1\over3}\|\Cal A'\|$ and
there is no pairs in $\Cal B'$ of distance ${\leqslant}a$ between its elements.

For the set $\Cal B{\leftrightharpoons}\Cal B'{\cup}\{I\}$, one has
$\|\Cal B\|{\geqslant}{1\over3}\|\Cal A\|$ and there is no ``$a$-close'' pairs${\subset}\Cal B$.\end{proof}
\subsection{End of the proof of Theorem 2: an estimate for the Van Kampen area of $\beta$}
By the corollary of Section 6.2, we have $\s{a}Q_I>\omega{\cdot}{1\over20}d|I|$ for
each segient $I{\in}\Cal A_d$.

By the lemma of Section 6.3, there exists $\Cal B{\subset}\Cal A_d$ such that
$\|\Cal B\|{\geqslant}{1\over3}\|\Cal A_d\|$. So  the distance between any two distinct
segments${\in}\Cal B$ is greater than ${4\over5}d$ (see Proposition \ref{p9}).

Since the ``lateral'' sides of each qurdrilateral $Q_I$, $I{\in}\Cal A_d$ have length ${1\over5}d$,
the lateral sides of distinct quadrilaterals $Q_I,Q_J$ ($I,J{\in}\Cal B$) have empty intersection (see the picture below).\par
\begin{picture}(0,120)(-1,-31)
\put(-14,-70){
\pdfximage{\mpPath 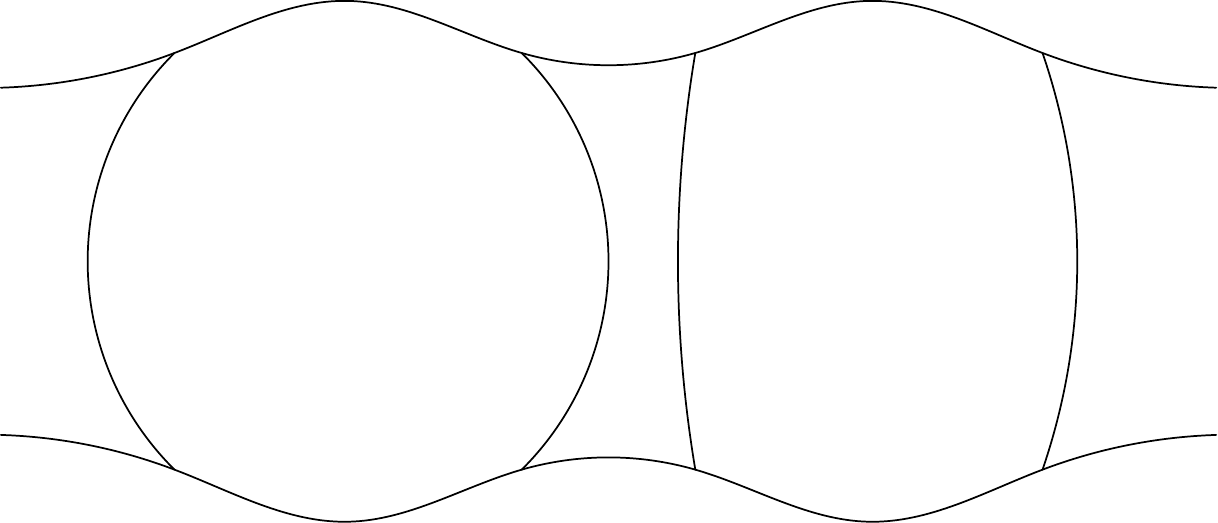}\pdfrefximage\pdflastximage}
\put(91,8){\makebox(0,0)[cc]{$Q_I$}}
\put(245,8){\makebox(0,0)[cc]{$Q_J$}}
\put(167,68){\makebox(0,0)[cc]{$\gamma$}}
\put(167,-59){\makebox(0,0)[cc]{$\gamma'$}}
\put(90,-78){\vector(1,0){51}}
\put(92,-78){\vector(-1,0){52}}
\put(91,-85){\makebox(0,0)[cc]{$I$}}
\put(244,-78){\vector(1,0){49}}
\put(246,-78){\vector(-1,0){57}}
\put(245,-85){\makebox(0,0)[cc]{$J$}}
\end{picture}
\vskip60pt
By the Jordan curve theorem, the pieces of the Van Kampen diagram bounded by $Q_I$
 and $Q_J$ have empty intersection.

Thus $\s{a}(\beta){\geqslant}\sum_{I\in\Cal B}\s{a}(Q_I)>\omega{\cdot}{1\over20}d{\cdot}\|\Cal B\|\geqslant
\omega{\cdot}{1\over60}d\|\Cal A_d\|\overset{(7),(8)}\geqslant{1\over60}\omega\varepsilon d{\cdot}|\beta|$.\hfil\penalty-10000
So $\Frac{\s{a}(\beta)}{|\beta|}>{1\over60}\omega\varepsilon d$.

Since the parameter $d$ can take arbitrary large values we have
a contradiction with the assumption of Theorem 2.

Theorem 2 is proved.

\vskip3pt
\sc Remark\rm. Theorem 2 can be stated for arbitrary graphs (not only for Cayley graphs of groups).
Instead of the Van Kampen area one can use the function
$I(P,\rho)$ of \cite{Bo}. Our proof is valid in this case too.


\begin{thebibliography}{GGPY21}
\bibitem{Anc} A. Ancona\sl. Positive harmonic functions and hyperbolicity\rm. In: Potential Theory
--- Surveys and Problems (Prague, 1987). Lecture Notes in Math. 1344, p. 123. Springer Berlin.
\bibitem{Bo}
B. H. Bowditch\sl, Notes on Gromov's hyperbolicity criterion for path-metric spaces\rm. In:
E. Ghys, A. Haefliger, A. Verjovsky (editors) \sl
Group Theory from a Geometric Viewpoint\rm. World Scientific. 1991, pp. 64--168.
\bibitem{GGPY} I. Gekhtman, V. Gerasimov, L. Potyagailo and W. Yang\sl.
Martin boundary covers Floyd boundary\rm. Invent. Math. \bf 223 \rm(1921) 759--809.
\bibitem{Gers}
S. M. Gersten\sl, Problems on automatic groups\rm,
Algorithms and Classification in Combinatorial Group
Theory (Gilbert Baumslag and Charles F. Miller, eds.), Springer New York, New York, NY, 1992,
pp. 225--232.
\bibitem{Gr87}
M. Gromov\sl, Hyperbolic groups\rm, in:
``Essays in Group Theory'' (ed. S.~M.~Gersten) M.S.R.I. Publications No.~8,
Springe-Verlag (1987) 75--263.
\bibitem{LSh}
R. C. Lyndon and P. E. Shupp\sl, Combinatorial Group Theory\rm. Springer-Verlag, 1977.
\bibitem{Pap}
P. Papasoglu\sl. Strongly automatic groups are hyperbolic\rm. Invent. Math. 121, 323--334 (1995).
\end{thebibliography}
\end{document}